\documentclass{amsart}

\usepackage[margin=1in]{geometry}
\usepackage{amssymb}
\usepackage{url}

\theoremstyle{plain}
 \newtheorem{theorem}{Theorem}
 \newtheorem{lemma}[theorem]{Lemma}
 \newtheorem{proposition}[theorem]{Proposition}
 \newtheorem{corollary}[theorem]{Corollary}
 
 \newtheorem{problem}{Problem}
 
\theoremstyle{remark}
 \newtheorem{example}[theorem]{Example}
 \newtheorem{remark}[theorem]{Remark}

\numberwithin{equation}{section}

\newcommand{\cp}{\sim_p}
\newcommand{\VVV}{\mathcal{V}}
\newcommand{\WWW}{\mathcal{W}}
\newcommand{\EEE}{\mathcal{E}}
\newcommand{\by}[1]{\overset{\eqref{#1}}{=}}  

\begin{document}

\title{Variants of epigroups and primary conjugacy}
\author{Maria Borralho}
\author{Michael Kinyon$^*$}
\thanks{${}^*$ Partially supported by Simons Foundation Collaboration Grant 359872 and
by Fundaç\~{a}o para a Ci\^{e}ncia e a Tecnologia (Portuguese Foundation for Science and Technology)
grant PTDC/MAT-PUR/31174/2017.}
\address[Borralho]{Universidade Aberta, R. Escola Polit\'{e}cnica, 147, 1269-001 Lisboa, Portugal}
\address[Borralho]{CEMAT-CI\^{E}NCIAS, Departamento de Matem\'{a}tica, Faculdade de Ci\^{e}ncias, Universidade de Lisboa, 1749-016, Lisboa, Portugal}
\address[Kinyon]{Department of Mathematics, University of Denver, Denver, CO 80208, USA}
\address[Kinyon]{Centre for Mathematics and Applications, Faculdade de Ci\^{e}ncias e Tecnologia,
Universidade Nova de Lisboa, Campus da Caparica, 2829-516 Caparica, PT}
\email[Borralho]{mfborralho@gmail.com}
\email[Kinyon]{mkinyon@du.edu}



\begin{abstract}
In a semigroup $S$ with fixed $c\in S$, one can construct a new semigroup $(S,\cdot_c)$ called a \emph{variant} by defining $x\cdot_c y:=xcy$. Elements $a,b\in S$ are \emph{primarily conjugate} if there exist $x,y\in S^1$ such that $a=xy, b=yx$. This coincides with the usual conjugacy in groups, but is not transitive in general semigroups. Ara\'{u}jo \emph{et al.} proved that transitivity holds in a variety $\WWW$ of epigroups containing all completely regular semigroups and their variants, and asked if transitivity holds for all variants of semigroups in $\WWW$. We answer this affirmatively as part of a study of varieties and variants of epigroups.
\end{abstract}

\maketitle

\section{Introduction}
\label{Sec:Intro}

Let $S$ be a semigroup. Given $c\in S$ we can define a new binary operation $\cdot_c$
on $S$ by
\begin{equation}\label{Eq:variant}
a\cdot_c b = acb
\end{equation}
for all $a,b\in S$. The operation $\cdot_c$ is clearly associative, and the semigroup
$(S,\cdot_c)$ and called the \emph{variant} of $S$ at $c$ (see
\cite{Hi86} and also \cite{dolinka,dolinka2,khan,KuMal,Jong,Maz,pei,Tsyaputa}
and (\cite{gm}, Chapter 13)).

Elements $a,b$ of a semigroup $S$ are said to be \emph{primarily conjugate}, denoted
$a \cp b$, if there exist $x,y\in S^1$ such that $a=xy$ and $b=yx$. Here as usual,
$S^1$ denotes $S$ if $S$ is a monoid; otherwise $S^1 = S\cup \{1\}$ where $1$ is an
adjoined identity element. Primary
conjugacy is reflexive and symmetric, but it is not transitive in general. The transitive closure
$\cp^{\ast}$ of $\cp$ can be considered to be a conjugacy relation in general semigroups
\cite{Hi06,KuMa07,KuMa09}. Primary conjugacy is transitive in
groups (where it coincides with the usual notion of conjugacy) and free semigroups \cite{La79}.
To describe additional classes where primary conjugacy is known to be transitive, we must
first recall the notion of epigroup.

An element $a$ of a semigroup $S$ is an \emph{epigroup element} (also known as a \emph{group-bound element})
if there exists a positive integer $n$ such that $a^n$ belongs to a subgroup of $S$,
that is, the $\mathcal{H}$-class $H_{a^n}$ of $a^n$ is a group. The smallest such $n$ is the \emph{index} of $a$.
If $H_a$ itself is a group, that is, if $a$ has index $1$, then $a$ is said to be \emph{completely regular}.
If we let $e$ denote the identity element of $H_{a^n}$, then $ae$ is in $H_{a^n}$ and we define the
\emph{pseudoinverse} $a'$ of $a$ by $a'=(ae)^{-1}$, where $(ae)^{-1}$ denotes the inverse of $ae$
in the group $H_{a^n}$ \cite[(2.1)]{Shevrin}. If every element of a semigroup is an epigroup element,
then the semigroup itself is said to be an \emph{epigroup}, and if every element is completely regular,
then the semigroup is said to be \emph{completely regular}. Every finite semigroup, and in fact every
periodic semigroup, is an epigroup. Following Petrich and Reilly \cite{PeRe99} for completely regular semigroups
and Shevrin \cite{Shevrin2,Shevrin} for epigroups, it is now customary to view an epigroup $(S,\cdot)$ as a \emph{unary}
semigroup $(S,\cdot,{}')$ where $x\mapsto x'$ is the map sending each element to its pseudoinverse.
We will make considerable use of the following identities which hold in all epigroups \cite{Shevrin}:
\begin{align}
  x'xx' &= x'\,,        \label{Eq:epi1} \\
  xx' &= x'x\,,         \label{Eq:epi2} \\
  x''' &= x'\,,         \label{Eq:epi3} \\
  xx'x &= x''\,,        \label{Eq:epi4} \\
  (xy)'x &= x(yx)'\,,   \label{Eq:epi5} \\
  (x^p)' &= (x')^p\,,   \label{Eq:epi6}
\end{align}
for any prime $p$.

For each $n\in \mathbb{N}$, let $\EEE_n$ denote the variety (equational class) of all unary semigroups
$(S,\cdot,{}')$ satisfying \eqref{Eq:epi1}, \eqref{Eq:epi2} and $x^{n+1} x' = x^n$.
Each $\EEE_n$ is a variety of epigroups, and the inclusions $\EEE_n\subset \EEE_{n+1}$ hold for all $n$.
Every finite semigroup, considered as an epigroup, is contained in some $\EEE_n$. $\EEE_1$ is the variety of completely regular semigroups.

The following observation will be useful later.

\begin{lemma}\label{Lem:EEE_alt}
For each $n\in \mathbb{N}$, the variety $\EEE_n$ is precisely the variety of unary semigroups satisfying \eqref{Eq:epi1},
\eqref{Eq:epi2} and $x^{n-1}x'' = x^n$.
\end{lemma}
\begin{proof}
  If $S$ is an epigroup in $\EEE_n$, then $x^n = x^{n+1}x' = x^{n-1} xx'x = x^{n-1}x''$ using \eqref{Eq:epi2}
  and \eqref{Eq:epi4}. Conversely, suppose $S$ satisfies \eqref{Eq:epi1}, \eqref{Eq:epi2} and $x^{n-1}x'' = x^n$.
  Then $x^{n+1}x' = x^n x'' x' = x^{n-1} x'' x'' x' = x^{n-1} x'' x' x'' = x^{n-1} x'' = x^n$,
using \eqref{Eq:epi2} in the third equality and \eqref{Eq:epi1} in the fourth equality.
\end{proof}

Kudryavtseva \cite{Ku06} proved that the restriction of $\cp$ to the set of all completely regular
elements of a semigroup is transitive.
More recently, it was shown in \cite{ArKiKoMa14} that $\cp$ is transitive in all variants of
completely regular semigroups. Variants of completely regular semigroups are not, in general,
completely regular themselves; for example, if a completely regular semigroup has a zero $0$,
then the variant at $0$ is a null semigroup, which is not even regular.
This difficulty was circumvented in \cite{ArKiKoMa14} by introducing the following class $\WWW$ of semigroups:
\[
S\in \WWW \quad\iff\quad xy\text{ is completely regular for all }x,y\in S\,.
\]
Equivalently $\WWW$ consists of all semigroups $S$ such that the subsemigroup
$S^2 = \{ab\mid a,b\in S\}$ is completely regular. The class $\WWW$ includes
all completely regular semigroups and all null semigroups (semigroups satisfying
$xy = uv$ for all $x,y,u,w$). The following summarizes the relevant results of
\cite{ArKiKoMa14}.

\begin{proposition}
\label{Prp:summary}
\begin{enumerate}
    \item[]
    \item \textup{(}\cite{ArKiKoMa14}, Prp. 4.14\textup{)} $\WWW$ is the variety of epigroups in $\EEE_2$ satisfying the additional identity $(xy)'' = xy$.
    \item \textup{(}\cite{ArKiKoMa14}, Thm. 4.15\textup{)} If $S$ is a epigroup in $\WWW$, then $\cp$ is transitive in $S$;
    \item \textup{(}\cite{ArKiKoMa14}, Thm. 4.17\textup{)} Every variant of a completely regular semigroup is in $\WWW$;
    \item \textup{(}\cite{ArKiKoMa14}, Cor. 4.18\textup{)} If $S$ is a variant of a completely regular semigroup, then $\cp$ is transitive in $S$.
\end{enumerate}
\end{proposition}

Part (2) of this proposition had more to it, comparing $\cp$ with other notions of conjugation.
In the simplified form stated here, the result follows easily from Kudryavtseva's theorem \cite{Ku06}:
if $a\cp b$, $b\cp c$, and $a\neq b\neq c\neq a$, then there exist $x,y,u,v\in S$ such that
$a = xy$, $b = yx = uv$ and $c = vu$. Thus $a,b,c\in \WWW$ are completely regular,
so $a\cp c$.

We can slightly improve Proposition \ref{Prp:summary}(1) as follows.

\begin{lemma}\label{Lem:WWW_alt}
The variety $\WWW$ is precisely the variety of unary semigroups satisfying the identities
\eqref{Eq:epi1}, \eqref{Eq:epi2}, \eqref{Eq:epi4}, \eqref{Eq:epi6} (for $p=2$) and $(xy)'' = xy$.
\end{lemma}
\begin{proof}
 One implication follows from Proposition \ref{Prp:summary}(1), so suppose $(S,\cdot,{}')$
 is a unary semigroup satisfying the identities listed in the lemma. Then for all $x\in S$,
 \begin{alignat*}{3}
   x^3 x' &\by{Eq:epi1} x^3 \underbrace{x'x}x'
    && \by{Eq:epi2} x^4 x'x' \\
   & \by{Eq:epi6} x^2 \underbrace{x^2 (x^2)'}
    && \by{Eq:epi2} x^2 (x^2)' x^2 \\
   & \by{Eq:epi4} (x^2)''
    && = x^2\,.
 \end{alignat*}
Therefore $(S,\cdot,{}')$ lies in $\EEE_2$, hence in $\WWW$.
\end{proof}

The variety $\WWW$ has another characterization that was not mentioned in \cite{ArKiKoMa14}.

\begin{lemma}
\label{Lem:WWW}
Let $S$ be a semigroup. The following are equivalent:
\begin{enumerate}
\item\quad $S$ is an epigroup in $\mathcal{W}$.
\item\quad For each $c\in S$, the principal left ideal $Sc$ is a completely regular subsemigroup.
\item\quad For each $c\in S$, the principal right ideal $cS$ is a completely regular subsemigroup.
\end{enumerate}
\end{lemma}
\begin{proof}
An element of a semigroup is completely regular if and only if it lies in some
subgroup, so the desired equivalences follow from the definition of $\WWW$.
\end{proof}

In view of Lemma \ref{Lem:WWW}, we should also mention the kindred study in \cite{LCH}
of epigroups $S$ in which every local submonoid $eSe$ is completely regular.

\section{Main Results}
\label{Sec:main}

The key tool in the proof of Proposition \ref{Prp:summary}(3) was the following unary
operation:
\begin{equation}\label{Eq:new_unary}
  x^{\ast} = (xc)' x (cx)'\,. \tag{$\ast$}
\end{equation}
Indeed, if $(S,\cdot,{}')$ is completely
regular, then $(S,\cdot_c,{}^{\ast})$ is an epigroup in the variety $\WWW$.
However, \eqref{Eq:new_unary} was introduced in \cite{ArKiKoMa14} in
an \emph{ad hoc} fashion. To show that it is quite natural, we note that
\emph{an ideal of an epigroup is a subepigroup} \cite[Obs.~4]{Shevrin2}.
In particular, for each $c$ in an epigroup $S$, $Sc$ is a subepigroup.
Thus for any $x\in S$, the pseudoinverse $(xc)'$ must have the form
$yc$ for some $y\in S$. This is exactly what \eqref{Eq:new_unary} does
for us.

\begin{lemma}
Let $S$ be an epigroup and fix $c\in S$. For all $x\in S$,
\begin{equation}\label{Eq:xstarc}
  (xc)' = x^* c\,.
\end{equation}
\end{lemma}
\begin{proof}
We compute
\[
(xc)' = (xc)'xc(xc)' = (xc)'(xc)'xc = (xc)x(cx)'\cdot c = x^* c\,,
\]
using \eqref{Eq:epi1}, \eqref{Eq:epi2} and \eqref{Eq:epi5}.
\end{proof}

If $(S,\cdot,{}')$ is an epigroup, we will refer to $(S,\cdot_c,{}^{\ast})$
as the \emph{unary variant} of $(S,\cdot,{}')$ at $c$. Proposition \ref{Prp:summary}(3) states
that if $(S,\cdot,{}')\in \EEE_1$ is completely regular, then $(S,\cdot_c,{}^{\ast})\in \WWW$.
Our first main result will both improve and extend this. First we must introduce a family of
varieties of unary semigroups. For each $n\in \mathbb{N}$, the variety $\VVV_n$ is defined
by associativity and the following identities: \eqref{Eq:epi1}, \eqref{Eq:epi2},
\begin{align}
  xy^{n-1} y'' &= xy^n  \label{Eq:VVV1} \\
  x'' x^{n-1}y &= x^n y \label{Eq:VVV2}
\end{align}
Setting $y = x$ in, say, \eqref{Eq:VVV1}, we see from Lemma \ref{Lem:EEE_alt} that $\VVV_n$ is
a variety of epigroups and in particular,

\begin{equation}\label{Eq:EVE}
\EEE_n\subseteq \VVV_n \subseteq \EEE_{n+1}\,.
\end{equation}

That every variant of an epigroup is an epigroup is easy to see, but what is not so obvious
is what happens to the pseudoinverse operation.
Our first main result clarifies this and also the role of the varieties $\VVV_n$.

\begin{theorem}\label{Thm:main1}
  Let $(S,\cdot,{}')$ be an epigroup. For each $c\in S$, the unary variant
  $(S,\cdot_c,{}^{\ast})$ is an epigroup. If $(S,\cdot,{}')\in \VVV_n$ for some $n > 0$,
  then $(S,\cdot_c,{}^{\ast})\in \VVV_n$. Therefore for $n\in \mathbb{N}$, the variety $\VVV_n$ is
  closed under taking variants.
\end{theorem}

\begin{corollary}\label{Cor:CR}
Let $(S,\cdot,{}')$ be a completely regular semigroup. For each $c\in S$, the unary variant
$(S,\cdot_c,{}^{\ast})$ lies in $\VVV_1$.
\end{corollary}

\begin{example}
Not every unary semigroup in $\VVV_1$ is a variant of a completely
regular semigroup. Using \textsc{Mace4}, we found that the smallest examples
have order $4$, and there are three of them up to isomorphism:
\[
\begin{array}{r|cccc}
\cdot & 0 & 1 & 2 & 3 \\
\hline
    0 & 1 & 1 & 1 & 1 \\
    1 & 1 & 1 & 1 & 1 \\
    2 & 1 & 1 & 2 & 1 \\
    3 & 1 & 1 & 1 & 3 \\
\end{array}
\qquad\qquad
\begin{array}{r|cccc}
\cdot & 0 & 1 & 2 & 3 \\
\hline
    0 & 1 & 1 & 2 & 2 \\
    1 & 1 & 1 & 2 & 2 \\
    2 & 2 & 2 & 2 & 2 \\
    3 & 2 & 2 & 2 & 3 \\
\end{array}
\qquad\qquad
\begin{array}{r|cccc}
\cdot & 0 & 1 & 2 & 3 \\
\hline
    0 & 2 & 2 & 1 & 1 \\
    1 & 2 & 2 & 1 & 1 \\
    2 & 1 & 1 & 2 & 2 \\
    3 & 1 & 1 & 2 & 3 \\
\end{array}
\]
The corresponding pseudoinverse operation is the same for all three
epigroups: $0'=1$ and $x'=x$ for $x=1,2,3$.
\end{example}

\begin{remark}\label{Rem:homom}
For an element $c$ of a semigroup $S$, the mapping  $\rho_c : S\to Sc; x\mapsto xc$
is a homomorphism from the variant $(S,\cdot_c)$ to $(Sc,\cdot)$ since
$(x\cdot_c y)c = xc\cdot yc$. If $S$ is also an epigroup, then as already noted, so is
$Sc$. Every semigroup homomorphism between epigroups is an epigroup homomorphism, but
\eqref{Eq:xstarc} shows more explicitly how $\rho_c$ preserves pseudoinverses.
\end{remark}

Comparing Proposition \ref{Prp:summary}(3) with Corollary \ref{Cor:CR} raises the question
of how the varieties $\VVV_1$ and $\WWW$ are related other than just the fact that both
contain $\EEE_1$.
Our second main result addresses this and its corollary connects this discussion
to the transitivity of $\cp$.

\begin{theorem}\label{Thm:main2}
  \begin{enumerate}
    \item[]
    \item\qquad $\VVV_1 \subset \WWW$;
    \item\qquad If $(S,\cdot,{}')$ is an epigroup in $\WWW$, then for each $c\in S$,
    the unary variant $(S,\cdot_c,{}^{\ast})$ lies in $\VVV_1$.
  \end{enumerate}
\end{theorem}

In particular, every variant of an epigroup in $\WWW$ actually lies in a
proper subvariety of $\WWW$, a stronger statement than the assertion that
$\WWW$ is closed under taking variants.

We can now give an affirmative answer to Problem 6.18 in \cite{ArKiKoMa14}.
Since $\cp$ is transitive in $\WWW$, the theorem immediately implies
the following.

\begin{corollary}\label{Cor:transitive}
Primary conjugacy $\cp$ is transitive in every variant of any epigroup in $\WWW$.
\end{corollary}

The proofs of Theorems \ref{Thm:main1} and \ref{Thm:main2} will be given in
{\S}\ref{Sec:proofs}. In {\S}\ref{Sec:problems}, we conclude with some open problems.

\section{Proofs}
\label{Sec:proofs}

Let $(S,\cdot,{}')$ be an epigroup and fix $c\in S$. To verify Theorem \ref{Thm:main1}, we
start with a few lemmas.

\begin{lemma}\label{Lem:main1-step1}
$(S,\cdot_c,{}^*)$ satisfies \eqref{Eq:epi1} and \eqref{Eq:epi2}.
\end{lemma}
\begin{proof}
First we compute
\begin{alignat*}{3}
x^*\cdot_c x\cdot_c x^* &= (xc)' \underbrace{x (cx)'} cxc (xc)' \underbrace{x (cx)'}
    && \by{Eq:epi5} (xc)'\underbrace{(xc)'xc}\underbrace{xc(xc)'}(xc)'x \\
& \by{Eq:epi2} (xc)'xc(xc)'\cdot (xc)'xc(xc)'\cdot x
    && \by{Eq:epi1} (xc)'\underbrace{(xc)'x} \\
& \by{Eq:epi5} (xc)'x(cx)'
    &&= x^*\,,
\end{alignat*}
which establishes \eqref{Eq:epi1}.

Next we have
\begin{alignat*}{3}
  x\cdot_c x^* &= \underbrace{xc(xc)'}x(cx)'
    && \by{Eq:epi2} (xc)'xc\underbrace{x(cx)'} \\
  & \by{Eq:epi5} (xc)'x\underbrace{c(xc)'}x
    && \by{Eq:epi5} (xc)'x(cx)'cx \\
  &= x^*\cdot_c x\,, &&
\end{alignat*}
which establishes \eqref{Eq:epi2}.
\end{proof}

We will denote powers of elements in $(S,\cdot_c)$ with parentheses in the exponent, that is,
$x^{(1)} = x$ and $x^{(n)} = x\cdot_c x^{(n-1)}$ for $n > 1$.

\begin{lemma}\label{Lem:main1-step2}
If $ca$ has index $n$ in $(S,\cdot,{}')$, then $a$ is an epigroup element of index $n$ or $n+1$
in $(S,\cdot_c,{}^{\ast})$.
\end{lemma}
\begin{proof}
For any $k > 0$, we have
$a^{(k+1)}\cdot_c a^* = (ac)^{k+1} a^* = a(ca)^k \underbrace{ca^*}
= a\underbrace{(ca)^k (ca)'}$ using \eqref{Eq:xstarc}, and
$a^{(k)} = a(ca)^k$. Thus if $ca$ has index $n$, then $a^{(n+2)}\cdot_c a^* = a^{(n+1)}$,
thus $a$ is an epigroup element of index at most $n+1$ in $(S,\cdot_c,{}^{\ast})$.
If $a$ has index $k\leq n+1$ in $(S,\cdot_c,{}^{\ast})$, then $a(ca)^k (ca)' = a(ca)^k$ and so
$(ca)^{k+1}(ca)' = (ca)^k$ and so $k \geq n$.
\end{proof}

\begin{lemma}\label{Lem:main1-step3}
For all $x\in S$,
\begin{equation}\label{Eq:starstar}
  cx^{**} = (cx)''\,.
\end{equation}
\end{lemma}
\begin{proof}
  Using \eqref{Eq:new_unary}, we compute
  \begin{alignat*}{3}
  c x^{**} &= c(x^*)*
    &&= \underbrace{c(x^* c)'}x^* (cx^*)' \\
  & \by{Eq:epi5} (cx^*)' cx^*(cx^*)'
    && \by{Eq:epi1} (cx^*)' \\
  &= (\underbrace{c(xc)'}x(cx)')'
    && \by{Eq:epi5} ((cx)'cx(cx)')'
  & \by{Eq:epi1} (cx)''\,.  \qedhere
  \end{alignat*}
\end{proof}

\begin{proof}[Proof of Theorem \ref{Thm:main1}]
  Assume $(S,\cdot,{}')$ is an epigroup. Then by Lemmas \ref{Lem:main1-step1} and \ref{Lem:main1-step2},
  $(S,\cdot_c,{}^{\ast})$ is also an epigroup.

  Suppose now that for some $n\in \mathbb{N}$, $(S,\cdot,{}')\in \VVV_n$. Then for all $x,y\in S$,
  \begin{alignat*}{3}
    x\cdot_c y^{(n-1)}\cdot_c y^{\ast\ast} &= x(cy)^{n-1}cy^{\ast\ast}
        &&\overset{\eqref{Eq:starstar}}{=} x(cy)^{n-1}(cy)''  \\
    &= x(cy)^n
        &&= x\cdot_c y^{(n)}\,,
  \end{alignat*}
  using $(S,\cdot,{}')\in \VVV_n$ in the third equality. Thus $(S,\cdot_c,{}^{\ast})\in \VVV_n$. This
  completes the proof.
\end{proof}

Now we turn to Theorem \ref{Thm:main2}.

\begin{lemma}\label{Lem:VinW}
$\VVV_1 \subset \WWW$
\end{lemma}
\begin{proof}
Fix $(S,\cdot,{}')\in \VVV_1$. We already know that $S$ is an epigroup in $\EEE_2$ by \eqref{Eq:EVE}
and so by Proposition \ref{Prp:summary}(1), we just need to verify the identity $(xy)'' = xy$.
We compute
\begin{alignat*}{4}
(xy)'' &\by{Eq:epi3} \underbrace{xy}(xy)'xy
    &&\by{Eq:VVV2} \underbrace{x''}y(xy)'xy
    && \by{Eq:epi1} x''x'\underbrace{x''y}(xy)'xy \\
&\by{Eq:VVV2} x''x'\underbrace{xy(xy)'xy}
    && \by{Eq:epi3} x''\underbrace{x'(xy)''}
    && \by{Eq:VVV1} \underbrace{x''x'}xy \\
&\by{Eq:VVV2} \underbrace{xx'x}y
    && \by{Eq:epi3} x''y
    && \by{Eq:VVV2} xy\,.
\end{alignat*}

To see that the inclusion is proper, consider the unary semigroup given by the multiplication table
\[
\begin{tabular}{c|cccc}
$\cdot$ & $0$ & $1$ & $2$ & $3$\\
\hline
    $0$ & $2$ & $3$ & $2$ & $2$ \\
    $1$ & $1$ & $1$ & $1$ & $1$ \\
    $2$ & $2$ & $2$ & $2$ & $2$ \\
    $3$ & $3$ & $3$ & $3$ & $3$
\end{tabular}
\]
and the unary operation $0' = 2$, $1'=1$, $2'=2$, $3'=3$. This is easily checked to be
an epigroup in $\mathcal{W}$ with~${}'$ as the pseudoinverse operation, but
$0''\cdot 1 = 2\cdot 1 = 2 \neq 3 = 0\cdot 1$, so \eqref{Eq:VVV2} does not hold.
\end{proof}

\begin{proof}[Proof of Theorem \ref{Thm:main2}]
  Lemma \ref{Lem:VinW} takes care of (1), so we need to prove (2).

  Let $(S,\cdot,{}')$ be an epigroup in $\WWW$ and fix $c\in S$.
  Since $\WWW\subseteq \EEE_2\subseteq \VVV_2$ (by Proposition \ref{Prp:summary}(1) and \eqref{Eq:EVE}),
  we know that the unary variant $(S,\cdot_c,{}^{\ast})$ is an epigroup in $\VVV_2$ (Theorem \ref{Thm:main1}).
  What remains is to prove that $(S,\cdot_c,{}^{\ast})$ satisfies \eqref{Eq:VVV1} and \eqref{Eq:VVV2} with $n=1$.
  We compute
  \[
  x\cdot_c y^{**} = xcy^{**} \by{Eq:starstar} x(cy)'' \by{Eq:VVV1} xcy = x\cdot_c y\,.
  \]
  This establishes \eqref{Eq:VVV1} in $(S,\cdot_c,{}^{\ast})$ and the proof of \eqref{Eq:VVV2} is similar.
\end{proof}

\section{Problems}
\label{Sec:problems}

Completely regular semigroups can be defined conceptually (unions of groups) or as unary semigroups
satisfying certain identities. The same is true of the variety $\WWW$; the conceptual definition
given in \cite{ArKiKoMa14} is that $S$ lies in $\WWW$ if $S^2$ is completely regular or
$\WWW$ can be defined as a variety of unary semigroups (Lemma \ref{Lem:WWW_alt}).

On the other hand, the epigroup varieties $\VVV_n$ only have a definition as unary semigroups.
Since they are closed under taking variants (Theorem \ref{Thm:main1}),
they are clearly interesting varieties interlacing the varieties $\EEE_n$ (see \eqref{Eq:EVE}).
Thus one might ask the following.

\begin{problem}
Is there a conceptual characterization of the varieties $\VVV_n$, or even just $\VVV_1$, analogous to
the characterizations of $\EEE_1$ and $\WWW$?
\end{problem}

From \eqref{Eq:EVE} and Theorem \ref{Thm:main2}, we have the following chain of varieties:
\[
\EEE_1 \subset \VVV_1 \subset \WWW \subset \EEE_2 \subset \VVV_2 \subset \EEE_3 \cdots\,.
\]

\begin{problem}
 Is there a natural family of varieties $\WWW_n$ interlacing
 the varieties in the chain above and such that $\WWW_1 = \WWW$?
 In addition, does the appropriate generalization of Theorem \ref{Thm:main2}(2) hold?
\end{problem}

An interesting direction for the study of the varieties $\VVV_n$
or $\WWW$ is to consider the subvarieties in which idempotents
commute, that is, so-called $E$-semigroups (see \cite{Almeida} and the references therein).
These are subvarieties because every idempotent in an epigroup has the form
$x'x$, and so $E$-epigroups are characterized by the identity $x'xy'y = y'yx'x$.

\begin{problem}
Study the varieties of $E$-epigroups in $\VVV_n$ or $\WWW$.
\end{problem}

Many classes of algebras can be characterized by forbidden subalgebras or forbidden divisors (quotients).
For example, distributive lattices can be characterized by two forbidden sublattices; similarly, stable
semigroups can be characterized by forbidding the bicyclic monoid as a subsemigroup \cite{La79};
see also \cite{forbidden} for another example. The considerations in the paper prompt the following natural
problems.

\begin{problem}
Can any of the inclusions of varieties considered here,
especially $\EEE_1\subset \VVV_1$ and $\VVV_1\subset \WWW$, be characterized by
forbidden subepigroups or forbidden epidivisors?
\end{problem}

Finally, returning to primary conjugacy, we rephrase two problems from \cite{ArKiKoMa14}
to the context of this paper. These were suggested to us by Prof. J. Ara\'{u}jo.

\begin{problem}
Characterize and enumerate primary conjugacy classes in various types of
transformation semigroups and their variants such as, for example,
those appearing in the problem list of \cite[{\S}6]{ak} or those appearing
in the list of transformation semigroups included in \cite{vhf}.
Especially interesting would be a characterization of primary conjugacy classes in
variants of centralizers of idempotents \cite{andre,arko2,arko1}, or in
variants of semigroups in which the group of units has a rich
structure \cite{aac,arbemics,arcameron,arcam,arcameron2,arsc}.
\end{problem}

In \cite{abk}, a problem on independence algebras was solved using their classification;
the same technique might perhaps be used to extend the
results in \cite{dolinka} and to solve the following.

\begin{problem}
Characterize $\cp$ in the variants of the endomorphism monoid of a finite
dimensional independence algebra.
\end{problem}

\section*{Acknowledgement}

This paper is part of the first author's dissertation for the PhD Program in Computational
Algebra at Universidade Aberta in Portugal. Most of the proofs were obtained with the
assistance of the automated theorem prover \textsc{Prover9} developed by McCune \cite{McCune}.

\end{document}